\documentclass[3p,sort&compress,number,preprint]{elsarticle}

\usepackage{amsmath,amsthm,amssymb}

\usepackage[misc]{ifsym}
\newcommand{\envelope}{(\kern1pt\Letter\kern1pt)}

\usepackage{tabularx}

\newtheorem{theorem}{Theorem}
\newtheorem{lemma}{Lemma}
\newtheorem{proposition}{Proposition}
\newtheorem{corollary}{Corollary}

\newtheorem{remark}{Remark}

\allowdisplaybreaks

\journal{}
\usepackage{fancyhdr}
\renewcommand{\headrulewidth}{0.0pt}

\usepackage{hyperref}

\usepackage{enumitem}

\DeclareMathOperator{\im}{ran}
\DeclareMathOperator*{\spanof}{span}
\DeclareMathOperator*{\supp}{supp}

\def\divdifsymb{\scaleobj{1.2}{\mathord{\kern.34em\vrule width.6pt height6.3pt depth-.28pt \kern-.34em\Delta}}}
\newcommand{\divdif}[1]{\divdifsymb({#1})}

\newcommand{\setof}[1]{\left\{{#1}\right\}}
\newcommand{\cset}[2]{\setof{#1\,:\,#2}}
\newcommand{\ivcc}[2]{\left[#1,#2\right]}
\newcommand{\ivco}[2]{\left[#1,#2\right)}

\newcommand{\ivoo}[2]{\left(#1,#2\right)}
\newcommand{\ivu}{\ivcc{0}{1}}

\newcommand{\abs}[1]{\left|#1\right|}
\newcommand{\norm}[1]{\left\Vert#1\right\Vert}

\newcommand{\normi}[1]{\left\Vert#1\right\Vert_\infty}
\newcommand{\normp}[1]{\left\Vert#1\right\Vert_p}
\newcommand{\normop}[1]{\left\Vert#1\right\Vert_{op}}

\newcommand{\spacecf}{C(\ivcc{0}{1})}
\newcommand{\spacedf}[1]{C^{{#1}}(\ivcc{0}{1})}

\newcommand{\spacepp}{\mathcal{S}(\Delta_n, k)}

\newcommand{\spacelpo}{L^p(\Omega)}
\newcommand{\spaceh}[2]{{X^{#2, #1}(\Omega)}}
\newcommand{\spacehzp}{\spaceh{0}{p}}
\newcommand{\spacehrp}{\spaceh{r}{p}}
\newcommand{\spacehzi}{\spaceh{0}{\infty}}
\newcommand{\spacehri}{\spaceh{r}{\infty}}
\newcommand{\spacew}[2]{{W^{#2,#1}(\Omega)}}

\newcommand{\spacewrp}{\spacew{r}{p}}
\newcommand{\mosrp}{\omega_{r,p}}

\newcommand{\kfuncrp}{K_{r,p}}

\newcommand{\seminrp}[1]{\abs{#1}_{r,p}}

\newcommand{\ballo}[2]{B({#1},{#2})}

\newcommand{\uballo}{\ballo{0}{1}}

\newcommand{\Nn}{{\mathbb N}}

\newcommand{\Rr}{{\mathbb R}}
\newcommand{\Cc}{{\mathbb C}}

\newcommand{\Sdk}{S_{{\Delta_n}, k}\,}

\newcommand{\Sdkk}{S_{{\Delta_n}, k+1}\,}

\newcommand{\ISdk}{V_{{\Delta_n}, k}\,}

\newcommand{\ISdkm}[1]{V_{{\Delta_n}, {k-#1}}\,}

\newcommand{\xik}[1]{\xi_{{{#1}}, k}}
\newcommand{\xikk}[1]{\xi_{{{#1}}, k+1}}
\newcommand{\xijk}{\xi_{j, k}}

\newcommand{\bspk}[1]{N_{{#1}, k}}

\newcommand{\bspmk}[1]{M_{{#1}, k}}

\newcommand{\meshmin}{\abs{\Delta_n}_{\mathrm{min}}}

\newcommand{\id}{I}

\newcommand{\linope}[1]{\mathcal{L}(#1)}
\newcommand{\boundope}[1]{\mathcal{L}(#1)}

\newcommand{\compope}[1]{\mathcal{K}(#1)}

\newcommand{\spec}[1]{\sigma(#1)}

\newcommand{\ie}{i.\,e.}

\newcommand*{\polkn}{e_{k}}
\newcommand*{\funkn}{\alpha^*_{k}}
\newcommand*{\poljn}[1]{e_{#1}}

\begin{document}

\begin{frontmatter}
  \title{Lower estimates for linear operators with smooth range}
  \author{Johannes~Nagler}
  \ead{johannes.nagler@uni-passau.de}

  \address{Fakult\"at f\"ur Informatik und Mathematik, Universit\"at Passau, Germany}

  \begin{abstract}
  We introduce a new method to prove lower estimates for the approximation error
  of general linear operators with smooth range in terms of classical moduli of
  smoothness and related $K$-functionals. In addition, we explicitly show how to
  derive lower estimates for positive linear operators with smooth range and
  apply this result to classical approximation operators. We finish with some
  remarks on the eigenvalues of Schoenberg's spline operator. 
  \end{abstract}

  
\end{frontmatter}
\fancypagestyle{pprintTitle}{%
\lhead{} \chead{}\rhead{}
\lfoot{}\cfoot{}\rfoot{{\footnotesize\itshape \hfill Preprint, \today}}
\renewcommand{\headrulewidth}{0.0pt}
}

\section{Introduction}
A convenient way to relate the decay rate of a sequence of approximations $T_n$
on a Banach space $X$ with the smoothness of the approximated function
$f \in X $ is to establish lower estimates in terms of classical moduli of
smoothness and related $K$-functionals: There
exists constants $C_1,C_2 > 0$ independent on $n$ such that
 \begin{equation*}
     C_1 \cdot \omega_r(f, \delta_n) \leq \norm{T_nf - f} \quad\text{and}\quad
     C_2 \cdot K_r(f, \delta_n^r) \leq \norm{T_nf - f}
   \end{equation*}
   holds for all $f \in X$ and $\delta_n \to 0$ for $n \to \infty$. Although
   there exist already several methods to derive such estimates, see e.g.
   \citeauthor{Knoop:1995} (\cite{Knoop:1995}, \cite{Knoop:1994}),
   \citet{MR1253439} and \citet{MR1272123}, these methods still require many
   restrictions and therefore are not applicable for a large number of linear
   operators.

   In this article, we introduce a new method to derive such lower estimates for
   arbitrary compact operators with smooth range satisfying a spectral property.
   The approximation operator can be defined on arbitrary bounded domains
   $\Omega \subset \Rr^d$ with a suitably smooth boundary. As underlying
   function spaces we consider the space of continuous functions and 
   $L^p$-spaces for $1 \leq p < \infty$. Consequently, we use the space of
   $r$-times continuously differentiable functions and classical Sobolev
   spaces as their corresponding smooth subspaces.

   We will prove lower estimates for linear operators based on a functional
   analytic framework depending on the fixed points of the operator and the
   smoothness of the range. The key idea is to estimate the semi-norm occuring
   in the $K$-functional by the approximation error using the convergence of the
   iterates of the operator. The only requirements of this approach are that the
   seminorms of the $K$-functionals are bounded on the range of the
   approximation operator and annihilate its fixed points. It will be shown that
   the degree of the modulus of smoothness or the used $K$-functional depends
   only on the smoothness of the range and the fixed points of $T$. Note that
   these results are an extension of the method shown in \cite{nagler2015joc},
   where lower estimates for Schoenberg's variation diminishing spline operator
   have been shown. Here, we establish a very flexible framework to prove lower
   estimates for very general linear approximation operators.

   We finish this article by discussing applications of these results. First, we
   show how to derive lower estimates for general positive linear operators with
   smooth range. Then, we provide concrete lower estimates for the Bernstein
   operator, the Kantorovi{\v c} operator, the Schoenberg operator and the
   integral Schoenberg operator. As the eigenvalues of the Schoenberg operator
   play an important role in the corresponding lower estimate, we give a
   characterization of them in the end of this article. 

\section{Preliminaries}
Let $\Omega \subset \Rr^d$ be a bounded domain with suitable smooth boundary. 

\subsection{Function spaces}
We use the multi-index notation of \citet{Schwartz:1950} to introduce
derivatives. Accordingly, we denote by $D^\alpha$ the differential
operator\index{Mixed differential operator}
\begin{equation*}
  D^\alpha = \frac{\partial^{\abs{\alpha}}}{\partial x_1^{\alpha_1} \partial x_2^{\alpha_2} \cdots \partial x_n^{\alpha_n}},
\end{equation*}
where $\alpha = (\alpha_1,\ldots,\alpha_n)$ is a multi-index with modulus
$\abs{\alpha} = \sum_{i=1}^n\alpha_i$. We denote by $C^r(\Omega)$ the space of
all complex valued functions $f$ that have continuous and bounded derivatives
$D^\alpha f$ up to order $r$, \ie, $\abs{\alpha} \leq r$. The norm on
$C^r(\Omega)$ is given by $\norm{f} := \sup_{\abs{\alpha} = r} \normi{D^\alpha f}$.


By $\spacelpo$, $1 \leq p < \infty$, we denote the space of Lebesgue measurable
functions defined on $\Omega$ whose $p$-th power is integrable with respect to
the measure $\mathrm{d}x = \mathrm{d}x_1\cdots\mathrm{d}x_n = \mathrm{d}\mu$.
The Sobolev space $\spacewrp$ corresponding to $\spacelpo$ contains all
functions $f \in \spacelpo$ where $D^\alpha f \in \spacelpo$ for all
orders $\abs{\alpha} \leq r$. 

To simplify notation and to combine the previously mentioned spaces,
we introduce the spaces $\spacehrp$ for $1 \leq p \leq \infty$ and
$r=0,1,2,\ldots$ as follows:
\begin{alignat*}{3}
    \spacehzp &:= \spacelpo,&\quad& 1 \leq p < \infty;&\qquad& \spacehzi := C(\Omega),\\
    \spacehrp &:= \spacewrp, &&1 \leq p < \infty; &&\spacehri := C^r(\Omega),
\end{alignat*}
Finally, we define the semi-norms 
\begin{equation}
  \label{eq:def_semi-norm}
  \seminrp{f} := \sup_{\abs{\alpha} = r} \normp{D^\alpha f}
\end{equation}
for all smooth functions $f \in \spacehrp$.

\subsection{Moduli of smoothness and $K$-functionals}
Now, we will introduce the modulus of smoothness and Peetre's $K$-functional for
the previously defined spaces according to \citet{Johnen:1976}. Let us denote
for $h \in \Rr^d$ the set
\begin{equation*}
  \Omega(h) := \cset{x \in \Omega}{x + th \in \Omega \quad\text{for}\quad 0 \leq t \leq 1}.
\end{equation*}
Then we define the $r$-th modulus of smoothness $\omega_{r,p}: \spacehzp \times \ivoo{0}{\infty} \to
  \ivco{0}{\infty}$, $1 \leq p \leq \infty$, as follows:
  \begin{equation*}
    \mosrp(f, t) :=
    \begin{cases}
      \normp{f}, & r=0\\
      \sup_{0<\abs{h}\leq t}\normp{\chi_{\Omega(rh)}\Delta_h^rf(x)}, & r=1,2,\ldots.
    \end{cases}
  \end{equation*}
  where $\Delta_h^r$ is the forward difference operator into direction $h\in \Rr^d$,
  \begin{equation*}
    \Delta_h^rf(x) = \sum_{l=0}^{r}(-1)^{r-l} \binom{r}{l}f(x + lh).
  \end{equation*}
  Similarly, the $K$-functional $K_{r,p}:\spacehzp\times \ivoo{0}{\infty} \to \ivco{0}{\infty}$,
  $1 \leq p \leq \infty$ is defnied on the
  spaces $\spacehrp$ as follows (\cite{peetre1968}, \cite{Johnen:1976}):
  \begin{equation}
    \label{eq:def_k-functional}
    \kfuncrp(f, t^r) := \inf\cset{\normp{f-g} + t^r\abs{g}_{p,r}}{g \in \spacehrp}.
  \end{equation}

As shown in \citet[Lem.~1]{Johnen:1976}, the modulus of smoothness can be bounded 
from above by the related $K$-functional in the following way: for all $0 < t <
\infty$ there holds
  \begin{equation}
    \label{eq:le_mos_kfunc}
    \mosrp(f, t) \leq 2^r \normp{f - g} + d^{r/2} t^r \seminrp{g},
  \end{equation}
  for $f \in \spacehzp$, $g \in \spacehrp$ and $1 \leq p \leq \infty$.
Moreover, the equivalence of the modulus of
smoothness to the $K$-functional have been shown, see
\citet{Butzer:1967} for the one-dimensional case and
\citet{Johnen:1976} for arbitrary Lipschitz domains.

\subsection{Projections and Iterates}
In order to prove lower estimates in a general setting, we will utilize the
convergence of the iterates to a projection operator. We will provide here the
necessary results that characterize this behaviour. To this end, let $X$ be a
complex Banach space and let us denote by $\linope{X}$ the set of all linear
operators on $X$. Note that the results shown here are also applicable on real
Banach spaces using a standard complexification scheme as outlined, e.g, in
\citet[pp.~7--16]{ruston1986}. 

In the following, we consider a bounded linear
contraction $T \in \linope{X}$, \ie, $\normop{T} \leq 1$.
\citet[Thm.~3.16]{Dunford:1943} has shown that the iterates converge to a
projection onto the corresponding fixed point space:
\begin{proposition}[Convergence of Iterates]
  \label{prop:le_limit_dunford_katznelson}
  Let $T \in \linope{X}$ be a compact operator such that $\normop{T^{m+1} - T^m}
  \to 0$ for $m \to \infty$. Then there exists $P \in \linope{X}$ with $P^2 =
  P$, and $P(X) = \ker(T - I)$ such that $T^m \to P$.
\end{proposition}
The necessary criteria, $\normop{T^{m+1} - T^m} \to 0$ for $m \to \infty$, has
been further characterized in the work of \citet[Thm.~1]{katznelson1986}, where they
provided a sufficient and necessary criterion based on the spectral location of
$T$.
\begin{proposition}[Spectral Location]
  \label{prop:le_katznelson}
  Let $T \in \linope{X}$ be a contraction. Then
  \begin{equation*}
    \lim_{m \to \infty}\normop{T^{m+1} - T^m} = 0
  \end{equation*}
  if and only if 
  \begin{equation}
    \label{eq:spectral_location}
    \sigma(T) \subset \uballo \cup \setof{1}.
  \end{equation}
\end{proposition}
The spectrum has to be contained in the unit ball with the only intersection at
$1$. 

Finally, it can be shown, that the convergence rate depends only on the second
largest spectral value in the modulus:
\begin{lemma}[Convergence Rate]
  \label{lem:iterates_convergence_rate}
  Let $T \in \boundope{X}$ be a compact operator with $r(T) = \normop{T} = 1$
  satisfying the spectral condition $\sigma(T) \subset \uballo \cup \setof{1}$.
  Define
  \begin{equation*}
    \gamma := \sup\cset{\abs{\gamma}}{\gamma \in \sigma(T) \setminus \setof{1}}.
  \end{equation*}
  Then there exists a constant $1 \leq C \leq \gamma^{-1}$, such that for all $m \in \Nn$
  \begin{equation*}
     \normop{T^m - P} \leq C \cdot \gamma^m,
  \end{equation*}
  where $P \in \compope{X}$ is the operator defined in
  \autoref{prop:le_limit_dunford_katznelson}.
\end{lemma}
\begin{proof}
  Using \citet[Thm.~3.16]{Dunford:1943}, we obtain the space decomposition $X =
  \ker(T - I) \oplus \im(T - I)$ and $\im(T - I)$ is closed. Accordingly, we
  decompose the operator $T$ into
  \begin{equation*}
    T = 
    \begin{pmatrix}
      I & 0\\
      0 & S
    \end{pmatrix} \in \boundope{\ker(T - I) \oplus \im(T - I)}.
  \end{equation*}
  Furthermore, we have that $\sigma(S) \subset \uballo$ and therefore we obtain
  $r(S) = \gamma < 1$. As $r(S) = \lim_{m\to \infty}\norm{S^m}^{1/m}$, we obtain that
  there exists a constant $1 \leq C \leq \gamma^{-1}$ such that
  \begin{equation*}
    \norm{S^m} \leq C \cdot \gamma^m
  \end{equation*}
  for every $m \in \Nn$. 
\end{proof}

\section{Lower estimates}
Let $\Omega \subset \Rr^d$ be a bounded domain with a suitably smooth boundary.
We consider now a sequence of linear operators $T_n$ defined on $\spacehzp$ with
smooth range $\im(T_n) \subset \spacehrp$ whose fixed point space $\ker(T_n -
I)$ is annihilated by every differential operator $D^\alpha$ of order $r$ that
is bounded on $\im(T_n)$. In this general setting, we will show that for all $s
\geq r$ and $n > 0$ there is $t_n > 0$ and there are constants $M_1,M_2 > 0$
independent of $n$ and $f \in \spacehzp$ such that
\begin{equation*}
  M_1 \cdot \omega_{s,p}(f,t_n) \leq \normp{T_n f - f}\quad\text{and}\quad
  M_2 \cdot K_{s,p}(f, t_n^s ) \leq \normp{T_n f - f}.
\end{equation*}
Here, $t_n \to 0$ for $n \to \infty$ provided that $\normp{f - T_nf} \to 0$.

In order to prove these estimates, we will consider the case where the smooth
function $g$ in \eqref{eq:def_k-functional} is replaced by the smooth
approximation $T_nf$. Then, we will estimate the semi-norm $\seminrp{T_nf} =
\sup\normp{D^\alpha T_n f}$ with respect to the approximation error $\normp{T_n
  f - f}$. The key concept of our approach is to use the limiting operator of
the iterates $T^n$ as shown in \autoref{prop:le_limit_dunford_katznelson}.
Recall that the compactness of the operators $T_n$ combined with a spectral
location will guarantee the existence of the limiting operator as seen in
\autoref{lem:iterates_convergence_rate} and \autoref{prop:le_katznelson}. With
this is mind, we can state the following lemma:

\begin{lemma}
  \label{lem:le_estimte_seminorm}
  Let $1 \leq p \leq \infty$ and let $T:\spacehzp \to \spacehzp$ be a
  compact contraction, \ie, $\normop{T} \leq 1$. Suppose
  \begin{enumerate}
  \item $\sigma(T) \subset \uballo \cup \setof{1}$,
  \item $\im(T) \subset \spacehrp$ for some positive integer $r$,
  \item $D^\alpha$ annihilates $\ker(T - I)$ for all $\alpha$ with $\abs{\alpha} = r$.
  \end{enumerate}
  Then for every $f \in \spacehzp$,
  \begin{equation*}
    \seminrp{T f} \leq \frac{\sup_{\abs{\alpha} = r}\normop{D^\alpha |_{\im(T)}}}{1-\gamma}\,\normp{T f - f},
  \end{equation*}
  where $\normop{D^\alpha |_{\im(T)}}$ is the operator norm of $D^\alpha$ on $\im(T)$ and
  \begin{equation*}
    \gamma := \sup\cset{\abs{\lambda}}{\lambda \in \spec{T} \text{~with~} \abs{\lambda} < 1}.
  \end{equation*}
\end{lemma}
\begin{proof}
  As $T$ is compact and exhibits the spectral property $\sigma(T) \subset
  \uballo \cup \setof{1}$, there exists a projection $P$ with $\im(P) = \ker(T -
  I)$ and according to \autoref{lem:iterates_convergence_rate} there exists a
  constant $0 \leq C \leq \gamma^{-1}$ such that
  \begin{equation*}
    \normop{T^m - P} \leq C \gamma^m
  \end{equation*}
  holds for all integers $m > 0$. As the range of $P$ is exactly the fixed point
  space of $T$, we have that $D^\alpha P = 0$ whenever $\abs{\alpha} \geq r$.

  Using these results we obtain
  \begin{align*}
    \seminrp{T f} &= \sup_{\abs{\alpha} = r}\normp{D^{\alpha}Tf}\\ &= \sup_{\abs{\alpha} = r}\normp{D^{\alpha} T f- D^{\alpha} T^2 f + D^{\alpha} T^2 f -D^{\alpha} T^3 f + \ldots}\\
    &\leq \sup_{\abs{\alpha} = r}\sum_{m=1}^\infty \normp{D^{\alpha} T^m (f- T f)}\\
    &\leq \normp{T f - f} \cdot \sup_{\abs{\alpha} = r}\sum_{m=1}^\infty \normop{D^{\alpha} T^m}\\
    &= \normp{T f - f} \cdot \sup_{\abs{\alpha} = r}\sum_{m=1}^\infty \normop{D^{\alpha}(T^m - P + P)}\\
    &=  \normp{Tf - f} \cdot \sup_{\abs{\alpha} = r}\sum_{m=1}^\infty \normop{D^{\alpha}(T^m - P)},\\
    \intertext{as $D^{\alpha}$ annihilates $\ker(T - I)$ and therefore,
      $D^{\alpha} P = 0$. By the boundedness of $D^{\alpha}$ on $\im(T)$ we get}
    \seminrp{T f} &\leq \normp{T f - f} \cdot \sup_{\abs{\alpha} = r}\normop{D^\alpha |_{\im(T)}}\sum_{m=1}^\infty \normop{T^m - P}\\
    &\leq \normp{T f - f} \cdot \sup_{\abs{\alpha} = r}\normop{D^\alpha |_{\im(T)} }\cdot \sum_{m=1}^\infty C\gamma^m.
\intertext{Using that $C \leq 1/\gamma$ the series reduces to a convergent geometric series and we conclude
      the proof with}
    \seminrp{T f} &\leq \normp{T f - f} \cdot \sup_{\abs{\alpha} = r}\normop{D^\alpha |_{\im(T)}} \cdot \sum_{m=0}^\infty \gamma^m\\
    &\leq \frac{\sup_{\abs{\alpha} = r}\normop{D^\alpha |_{\im(T)}}}{1-\gamma}\, \normp{T f - f} .
  \end{align*}
\end{proof}
Note that the third condition of \autoref{lem:le_estimte_seminorm} is reflected
in the shown estimate as for each $f \in \ker(T - I)$ we have that $\normp{Tf -
  f} = 0$ and $\seminrp{Tf} = 0$.

Using this lemma, we can state the main results of this article:
\begin{theorem}
  \label{thm:le_lower_bound_mos}
  Let $1 \leq p \leq \infty$ and let $T:\spacehzp \to \spacehzp$ be a compact 
  contraction that satisfies the following conditions:
  \begin{enumerate}
  \item $\sigma(T) \subset \uballo \cup \setof{1}$,
  \item $\im(T) \subset \spacehrp$ for some positive integer $r$,
  \item $D^\alpha$ annihilates $\ker(T - I)$ for all $\alpha$ with $\abs{\alpha} = r$.
  \end{enumerate}
  Then
  \begin{equation*}
    \mosrp(f, t) \leq \left(2^r + d^{r/2}t^r\frac{\sup_{\abs{\alpha} = r}\normop{D^\alpha|_{\im(T)}}}{1-\gamma}\right) \cdot \normp{Tf - f}
  \end{equation*}
  and
  \begin{equation*}
    \kfuncrp(f, t^r) \leq \left(1 + t^r\frac{\sup_{\abs{\alpha} = r}\normop{D^\alpha|_{\im(T)}}}{1-\gamma}\right) \cdot \normp{Tf - f}
  \end{equation*}  
   holds for all $t \in \ivoo{0}{\infty}$, 
   where $\gamma := \sup\cset{\abs{\lambda}}{\lambda \in \sigma(T) \text{ with } \lambda \neq 1}$.
\end{theorem}
\begin{proof}
  We apply \eqref{eq:le_mos_kfunc} and \autoref{lem:le_estimte_seminorm} to
  obtain the stated result.
\end{proof}

\begin{corollary}
  \label{cor:modulus_of_smoothnes_lower_bound_sharp}
  Let $(T_n)$ be a sequence of continuous linear operators on $\spacehzp$ that
  satisfies the conditions of \autoref{thm:le_lower_bound_mos}. Besides, we
  assume that $\normp{T_nf - f} \to 0$ holds for all $f \in \spacehzp$ if $n$
  tends to infinity.

  Then, with setting
  $\gamma_n := \sup\cset{\abs{\lambda}}{\lambda \in \sigma(T_n)
    \setminus\setof{1}}$ the uniform lower estimates
  \begin{equation*}
     \mosrp(f, \delta_n) \leq (2^r + d^{r/2})\cdot \normp{T_nf - f} \quad\text{and}\quad
     \kfuncrp(f, \delta_n^r) \leq 2\cdot \normp{T_nf - f}
   \end{equation*}
   holds, where
   \begin{equation*}
     \delta_n = \left( \frac{1-\gamma_n}{\sup_{\abs{\alpha} = r}\normop{D^\alpha|_{\im(T)}}} \right)^{1/r}
   \end{equation*}
    and $\delta_n \to 0$ if $n$ tends to infinity.
\end{corollary}

\begin{remark}
  The property that $\delta_n \to 0$ if $n$ tends to infinity follows by $\normp{T_nf -f}$ 
  for $f \in \spacecf$.
  To assure that this property holds there are the following two options. 
  Either the second largest eigenvalue tends in the modulus to one, \ie,
  \begin{equation*}
    \gamma_n \to 1
  \end{equation*}
  which is satisfied as $T_n$ converges against the identity $I$ in the strong operator 
  topology, or $\sup_{\abs{\alpha}=r}\normop{D^\alpha|_{\im(T)}} \to \infty$. 
\end{remark}

Finally, we want to outline a generalization to derive lower estimates
for a sequence of linear operators $(T_n)_{n \in \Nn}$ on arbitrary
Banach spaces based on the $K$-functional where smoothness of the
range is not necessary. The conditions depend on the underlying
semi-norms defined on the range of $T_n$. Accordingly, the semi-norms have to
annihilate the fixed points of $T_n$ and are bounded on the range of $T_n$.
 \begin{theorem}
   Let $(X_1, \norm{\cdot}_{X_1})$ be a Banach space and $(X_2, \abs{\cdot}_{X_2})$ 
   be a quasi Banach space with $X_2 \subset X_1$. Consider a sequence $T_n: X_1 \to X_2$  
   of compact contractions, such that the following conditions hold:
   \begin{enumerate}
    \item $\spec{T_n} \subset \uballo \cup \setof{1}$,
    \item the semi-norm $\abs{\cdot}_{X_2}$ annihilates $\ker(T_n - \id)$.
    \end{enumerate}
    Then
    \begin{equation*}
      \frac1{2}\cdot \inf_{g \in X_2} \left(\normp{f - g} + \delta_n^r\abs{g}_{X_2}\right) \leq  \normp{T_n f - f},
    \end{equation*}
    where
    \begin{equation*}
     \delta_n = \left( \frac{1-\gamma_n}{\sup_{f \in X_2, \norm{f}_{X_2}=1}\abs{T_n f}_{X_2}} \right)^{1/r}.
   \end{equation*}
 \end{theorem}
 \begin{proof}
   Follows directly along the lines of the proof of \autoref{thm:le_lower_bound_mos}.
 \end{proof}

\section{Applications to Positive Linear Operators}
We conclude this chapter with concrete examples. First we prove lower estimates
for general positive linear operators. Afterwards, we prove give concrete
estimates for the Bernstein operator, the Kantorovi{\v c} operator, the the
Schoenberg operator and the integral Schoenberg
operator. 

\subsection{Lower estimates for general positive finite-rank operators}
In the following, let $\Omega = [0,1]^d$, thus $\spacehrp$ contains the constant
function $1$ with $\normp{1} = 1$. We consider a sequence of positive
finite-rank operator $T_n:\spacehzp \to \spacehzp$,
\begin{equation}
  \label{eq:le_discrete_operator}
  T_n f = \sum_{k=1}^n \funkn(f)\polkn,\qquad f \in \spacehzp,
\end{equation}
where $\poljn{1},\ldots,\poljn{n} \in \spacehrp$ are linearly independent,
smooth positive functions that form a partition of unity; $\funkn$ are positive
linear functionals satisfying $\norm{\funkn} = \funkn(1) = 1$ and
$\funkn(\polkn) > 0$ for $k \in \setof{1,\ldots, n}$. It has been shown in
\cite{nagler2014fr}, that the spectrum of $T_n$ is characterized by
\begin{equation*}
  \sigma(T_n)  \subset \uballo \cup \setof{1}
\end{equation*}
and $1$ is an eigenvalue of $T_n$ due to the partition of unity property. Thus,
to prove lower estimates with the technique shown in this chapter, only last
condition have to be checked. Thus, we can restate
\autoref{cor:modulus_of_smoothnes_lower_bound_sharp} as follows:
\begin{corollary}
  \label{cor:modulus_of_smoothnes_lower_bound_discr}
  Let $(T_n)$ be a sequence of continuous linear operators on $\spacehzp$ of the
  form \eqref{eq:le_discrete_operator} such that $\normp{T_nf - f} \to 0$ holds
  for all $f \in \spacehzp$ if $n$ tends to infinity. Let us denote
  $$\gamma_n := \sup\cset{\abs{\lambda}}{\lambda \in \sigma(T_n) \text{~with~}
    \lambda \neq 1}.$$
  If every differential operator of order $r$ annihilates $\ker(T_n - I)$
  then the approximation error can be bounded from below by
  \begin{equation*}
     \mosrp(f, \delta_n) \leq (2^r + 1)\cdot \normp{T_nf - f} \quad\text{and}\quad
     \kfuncrp(f, \delta_n^r) \leq 2\cdot \normp{T_nf - f},
   \end{equation*}
   where
   \begin{equation*}
     \delta_n = \left( \frac{1-\gamma_n}{\sup_{\abs{\alpha} = r}\normop{D^\alpha|_{\im(T_n)}}} \right)^{1/r}. 
   \end{equation*}
    and $\delta_n \to 0$ if $n$ tends to infinity.
\end{corollary}

\subsection{Lower estimate for the Bernstein operator}
Let $B_n:\spacecf \to \spacecf$ be the Bernstein operator of order $n > 0$ defined by
\begin{equation*}
  B_nf(x) = \sum_{k=0}^{n}f\left(\frac{k}{n}\right) \binom{n}{k}x^k(1-x)^{n-k}.
\end{equation*}
It is well known, see e.g. \citet{Lorentz1986}, that this operator can reproduce
constant and linear functions and interpolates at the endpoints of the unit
interval. Therefore
\begin{align*}
  \ker(B_n - I) &= \spanof(1, x)
\end{align*}
and $D \ker(B_n - I) = 0$.
As shown by \citet{Calugareanu1966}, the eigenvalues
$(\lambda_{k,n})$ of $B_n$ are explicitly known for $k \in \setof{0,\ldots,n}$
by
\begin{equation*}
  \lambda_{k,n} = \frac{n!}{(n-k)!}\frac1{n^k}.
\end{equation*}
A comprehensive discussion on the corresponding eigenfunctions can be found in
the work of \citet{shaun2000}. Clearly, we have $\sigma(B_n) \subset \uballo
\cup \setof{1}$, as
\begin{equation*}
  1 = \lambda_{0,n} = \lambda_{1,n} > \lambda_{2,n} > \ldots > \lambda_{n,n} = \frac{n!}{n^n},
\end{equation*}
while this property also follows by \cite{nagler2014fr}. The second largest
eigenvalue $\gamma_n$ of $B_n$ is $\gamma_n := \lambda_{2,n} = \frac{n-1}{n}$.

The range of the Bernstein operator is given by the space of all polynomials
with degree at most $n$. Thus, we obtain for $r < n$  the following upper bound for the
operator norm of $D^r$ on $\im(B_n)$ using the representation of $D^rB_nf$ in \citet[p.24]{Lorentz1986}:
\begin{equation*}
  \normop{D^r} \leq \frac{2^rn!}{(n-r)!}.
\end{equation*}
Finally, we obtain with \autoref{thm:le_lower_bound_mos} the lower estimate
\begin{equation*}
  \omega_{r}(f, t) \leq \left(2^r + t^r\frac{\frac{2^r \cdot n!}{(n-r)!}}{\frac{1}{n}}\right) \cdot \normi{Tf - f}
                  \leq 2^r \left(1 + n^{r+1} t^r\right) \cdot \normi{Tf - f}
\end{equation*}
for all $t \in \ivoo{0}{\infty}$. For the case $r = 2$, we derive
accordingly the following uniform estimate:
\begin{corollary}  
  The approximation error of the  Bernstein operator $B_n$ can be uniformly bounded for all $f \in \spacecf$ by
  \begin{equation*}
    \frac{1}{8} \omega_2\left(f,n^{-3/2}\right) \leq \normi{B_n f - f}, \qquad n \to \infty.
  \end{equation*}
\end{corollary}
\begin{remark}
  Compared to the known lower estimate using the Ditzian-Totik modulus of
  smoothness as shown by \citet{ditzian1987} and \citet{Knoop:1994} one would
  expect a decay rate of $n^{-1/2}$. The question arises, whether sharper
  estimates used in the proof can lead to this optimal decay rate or if this is
  already the best possible lower estimate for the classical modulus of smoothness.
\end{remark}

\subsection{Lower estimate for the Kantorovi{\v c} operator}
Let us consider the Kantorovi{\v c} operator $K_n:L^1(\ivu) \to \spacecf$,
\begin{equation*}
  K_nf(x) = (n+1)\sum_{k=0}^n \binom{n}{k}x^k(1-x)^{n-k}\int_{\frac{k}{n+1}}^{\frac{k+1}{n+1}}f(t)\mathrm{d}t,\qquad x \in \ivu
\end{equation*}
see \citet{kantorovich1930}. This operator has a direct relation to the
Bernstein operator in the following way \cite[p.30]{Lorentz1986}:
\begin{equation}
  \label{eq:rel-bernstein-kantorovich}
  K_n(Df) = D(B_{n+1}f), \qquad \text{for all } f \in \spacedf{1}.
\end{equation}
Besides, we have that $\ker(K_n - I) = \spanof\setof{1}$. Infact, $D1 = 0$,
hence the differential operator $D$ annihilates $\ker(K_n - I)$. Besides, $D$ is
bounded on $\im(K_n)$ in the same way as the Bernstein operator:
\begin{equation*}
  \normp{D K_n f(x)} = \normp{D^2 B_{n+1} F(x)} \leq \normop{D^2|_{\im(B_{n+1})}} \norm{f}_1,
\end{equation*}
where we have used \eqref{eq:rel-bernstein-kantorovich} and $F(x) = \int_0^xf(t)\mathrm{d}t$. Therefore,
\begin{equation*}
  \normop{D} \leq \normop{D^2|_{\im(B_{n+1})}} = \frac{4(n+1)!}{(n+1-2)!} = 4(n^2+n).
\end{equation*}
holds. Combining these results with \autoref{thm:le_lower_bound_mos} we can
state the lower estimate
\begin{equation*}
  \omega_{1,p}(f, t) \leq \left(2 + t\frac{4(n^2+n)}{\frac{1}{n}}\right) \cdot \normi{Tf - f}
  \leq  \left(2 + 4(n^{3}+n^2) t\right) \cdot \normi{Tf - f}
\end{equation*}
for all $t \in \ivoo{0}{\infty}$. Consequently, we get the following uniform estimate:
\begin{corollary}  
  The approximation error of the Kantorovi{\v c} operator $K_n$ can be uniformly bounded from below by
  \begin{equation*}
    \frac{1}{6}\, \omega_{1,p}\left(f, \frac1{n^3 + n^2}\right) \leq \normi{K_n f - f}, \qquad n \to \infty,
  \end{equation*}
  for all $f \in L^1(\ivu)$.
\end{corollary}
As in the case of the Bernstein-operator, we are not able to derive the optimal
lower estimate shown in \citet{MR1260362} with the Ditzian-Totik modulus of
continuity, but we could still provide an estimate with the classical modulus of
continuity.

\renewcommand{\divdif}[1]{[#1]}
\subsection{Lower estimate for the Schoenberg operator}
A lower estimate for the Schoenberg operator has already been shown in
\citet{nagler2015joc} using similar techniques. Thus we state here only the
results for the sake of completeness. To this end let $n > 0$ be an integer and
$\Delta_n = \setof{x_{j}}_{j=-k}^{n+k}$ be an extended knot sequence such
that
\begin{equation*}              
  0 = x_{-k} = x_0 < x_1 < \ldots < x_n = x_{n_k} = 1.
\end{equation*}
Accoding to \citet{Schoenberg:1965}, we consider the variation diminishing
spline operator $\Sdk:\spacecf \to \spacecf$ of degree $k$ with respect to the
knot sequence $\Delta_n$ for continuous functions $f$ by
\begin{equation*}
  \label{eq:vdo_schoenberg_defop}
  \Sdk f = \sum_{j=-k}^{n-1}f(\xi_{j,k})\bspk{j},
\end{equation*}
where $\xijk$ are the so called \emph{Greville nodes}, see the supplement in
\cite{Schoenberg:1965}, defined for all $j \in \setof{-k,\ldots,n-1}$ by
\begin{equation*}
  \label{eq:vdo_schoenberg_greville}
  \xi_{j,k} := \frac{x_{j+1} + \cdots + x_{j+k}}{k}.
\end{equation*}
The normalized B-splines $\bspk{j}$ are defined for all $j \in
\setof{-k,\ldots,n-1}$ and $x \in \ivu$ by
\begin{equation*}
  \bspk{j}(x) := (x_{j+k+1} - x_j)\divdif{x_j,\ldots,x_{j+k+1}}(\cdot - x)_+^k,
\end{equation*}
where $\divdif{x_j,\ldots,x_{j+k+1}}$ denotes the divided difference
operator and $x^k_+$ denotes the truncated power function. 
We define the minimal mesh gauge as
\begin{equation*}
\meshmin := \min\cset{(x_{j+1,k} - x_{j,k})}{j \in \setof{0, \ldots, n-1}}
\end{equation*}
and $\gamma_{{\Delta_n}, k} := \sup\cset{\lambda \in \Cc}{\lambda \in \sigma(\Sdk) \setminus \setof{1}}$.
Then we can state the following lower estimate, see \cite[Cor.~2]{nagler2015joc}:
\begin{corollary}
  \label{thm:schoenberg_lower_bound}
  Let $f \in \spacecf$ and $k > r \geq 2$. Then 
  \begin{equation*}
    \frac1{2^{r + 1}}\omega_r\left(f, t(\Delta_n,k)\right) \leq \normi{f-\Sdk f},
  \end{equation*}
  where
  \begin{equation*}
    t(\Delta_n, k) = \frac{\meshmin}{k}\cdot\left(\frac{1-\gamma_{{\Delta_n}, k}}{d_k}\right)^{1/r}.
  \end{equation*}
  Moreover, $t(\Delta_n,k) \to 0$ if the approximation error converges to zero.
\end{corollary}
In order to get concrete values, it would be very interesting to have an exact
representation of the eigenvalues of $\Sdk$.

\subsection{Lower estimate for the integral Schoenberg operator}
The integral Schoenberg operator is defined by
\begin{equation*}
  V_{{\Delta_n},k}f(x) := D \Sdkk F(x) = 
  \sum_{j=-k}^{n-1} \int_{\xikk{j-1}}^{\xikk{j}}f(t)\mathrm{d}t \frac{\bspk{j}(x)}{\xi_{j,k+1} - \xi_{j-1,k+1}},
\end{equation*}
where $F(x) = \int_0^xf(t)\mathrm{d}t$. More details are shown in
\citet{mueller1977}. We have that $\ker(V^m_{{\Delta_n},k} - I) =
\spanof\setof{1}$ and $D1 = 0$ holds. By \cite{nagler2014fr}, we can conclude that
\begin{equation*}
  \spec{V_{{\Delta_n},k}} \subset \uballo \cup \setof{1},
\end{equation*}
holds. The operator norm of the differential operator $D$, can be obtained
similarly to the Kantorovi{\v c} operator.
To this end, we utilize a similar relation as in
\eqref{eq:rel-bernstein-kantorovich} between the Schoenberg operator and its
counterpart for the $L^p$-spaces:
\begin{lemma}
  \label{lem:relation_sdk_isdk}
  For all \(f \in \spacedf{1}\) the relation
  \begin{equation*}
    D\Sdk f = \ISdkm{1} D f
  \end{equation*}
  holds.
\end{lemma}
\begin{proof}
  Follows directly by the definition of the integral Schoenberg operator, as 
  \begin{equation*}
    V_{{\Delta_n},k}f(x) = D \Sdkk \int_{0}^xf(t)\mathrm{d}t.
  \end{equation*}
  Then a simple calculation yields
  \begin{equation*}
    \ISdkm{1}Df(x) = D \Sdk \int_{0}^x Df(t) \mathrm{d}t = D \Sdk (f(x) - f(0)) = D \Sdk f(x).
  \end{equation*}
  In the last step, we used the linearity of $\Sdk$ and that $\Sdk$ can reproduce constants.
\end{proof}
Now we can use this relation between $\ISdk$ and $\Sdk$ to derive
\begin{equation}
  \label{eq:le_ISdk_diff_bound}
  \normp{D \ISdk f} = \normp{D^2 \Sdkk F} = \norm{D^2}_{op:\im({\Sdkk})} \norm{f}_1,
\end{equation}
where $F(x) = \int_0^xf(t)\mathrm{d}t$. Using \eqref{eq:le_ISdk_diff_bound} and
the shown operator norm of $D^2$ on $\im(\Sdkk)$, see \cite{nagler2015joc}, we
obtain the following bound on $\im(\ISdk)$:
\begin{equation*}
  \norm{D}_{op:\im({\ISdk})} \leq \norm{D^2}_{op:\im({\Sdkk})} = \left(\frac{2(k+1)}{\meshmin}\right)^{2}d_{k+1}.
\end{equation*}
As all conditions of \autoref{cor:modulus_of_smoothnes_lower_bound_sharp} are
satisfied, we can state the following lower estimates:
\begin{corollary}
  Lower estimates for the integral Schoenberg opeartor $V_{{\Delta_n},k}$ are given by
  \begin{equation*}
    \frac1{6} \omega_{1,p}(f,t(\Delta_n, k)) \leq \normp{V_{{\Delta_n},k}f - f},
  \end{equation*}
  where
  \begin{equation*}
    t(\Delta_n, k) = \frac{\meshmin^2}{(k+1)^2}\cdot\left(\frac{1-\gamma_{{\Delta_n}, k}}{d_{k+1}}\right).
  \end{equation*}
\end{corollary} 

\section{Remarks on the eigenvalues of the Schoenberg operator}
The eigenvalues of the Bernstein operator have been revealed already in
\citeyear{Calugareanu1966} by the Russian \citet{Calugareanu1966}. Up to
our knowledge results on the eigenvalues of the Schoenberg operator are not
known explicitly. In the following, we show that $1$ is a simple eigenvalue of
$V_{{\Delta_n},k}$ and all the other eigenvalues are distinct non-negative, real
numbers. Finally, we will show that the Schoenberg operator has the same
eigenvalues as $V_{{\Delta_n},k}$ with the exception that $1$ is not a simple
eigenvalue as the Schoenberg operator reproduces constants and linear
functions.

To simplify notation, we define the B-splines $\bspmk{j}$ for $j \in 
\setof{-k,\ldots, n-1}$ as in \cite{Curry:1966} by:
\begin{equation}
  \label{eq:def_bspmk}
  \bspmk{j}(x) := \frac{\bspk{j}(x)}{\xi_{j,k+1} - \xi_{j-1,k+1}}.
\end{equation}
Note that these functions are normalized to have integral one, i.e.   $\int_0^1
\bspmk{j}(x)\mathrm{d}x = 1$, and have finite support:
\begin{equation*}
  \supp \bspmk{j}(x) = \ivcc{x_j}{x_{j+k+1}}.
\end{equation*}
Using this notation, we state the following theorem:
\begin{theorem}
  The collocation matrix of the integral Schoenberg operator 
  with the normalized B-splines as defined in \eqref{eq:def_bspmk}
  \begin{equation}
    \label{eq:ISdk-collocation}
    \left(\int_{\xikk{i-1}}^{\xikk{i}} \bspmk{j}(t) \mathrm{d}t \right)_{ij}
  \end{equation}
  is an oscillatory matrix. Thus, all eigenvalues are distinct positive real numbers.
\end{theorem}
\begin{proof}
  Recall, that the Greville nodes $\xik{j}$ are defined as the knot averages as in \eqref{eq:vdo_schoenberg_greville} by
  \begin{equation*}
  \xi_{j,k} := \frac{x_{j+1} + \cdots + x_{j+k}}{k}.
  \end{equation*}
  First, note that the relations
  \begin{align*}
    &x_j < \xikk{j-1} < \xik{j} < \xikk{j} < x_{j+k+1},\\
    &x_{j+1} < \xikk{j} < \xik{j+1} < \xikk{j+1} < x_{j+k+2} 
  \end{align*}
  and
  \begin{align*}
    \supp \bspmk{j} = \ivcc{x_j}{x_{j+k+1}},\qquad \supp \bspmk{j+1} =
    \ivcc{x_{j+1}}{x_{j+k+2}}
  \end{align*}
  hold. From the continuity of $\bspmk{j}$ and $\bspmk{j+1}$ and the relations
  \begin{equation*}
    \bspmk{j}(\xik{j}) > 0, \quad \bspmk{j+1}(\xikk{j}) > 0,\quad \bspmk{j}(\xikk{j}) > 0,
  \end{equation*}
  we can follow that
  \begin{equation*}
    \int_{\xikk{j-1}}^{\xikk{j}} \bspmk{j}(t) \mathrm{d}t > 0, \int_{\xikk{j-1}}^{\xikk{j}} \bspmk{j+1}(t) \mathrm{d}t > 0, \text{ and}
    \int_{\xikk{j}}^{\xikk{j+1}} \bspmk{j}(t) \mathrm{d}t > 0
  \end{equation*}
  holds. Moreover, the matrix 
  \begin{equation*}
    \left(\int_{\xikk{i-1}}^{\xikk{i}} \bspmk{j}(t) \mathrm{d}t \right)_{ij}
  \end{equation*}
  is non-singular as the B-splines $\bspmk{j}(x)$ are linearly independent and
  so are the functionals $\int_{\xikk{i-1}}^{\xikk{i}}\cdot\, \mathrm{d}t$ due
  to their distinct support. Using the well known result of \cite[Thm.~10,
  p.100]{gantmacher2002}, which states that a totally positive matrix \(A \in
  \Rr^{n \times n}\) is oscillatory if and only if
  \(A\) is non-singular and  \(a_{i,i+1} > 0\), \(a_{i+1,i} > 0\) for $i \in
  \setof{1,\ldots,n-1}$,
  we can conclude that the collocation matrix is
  oscillatory.  By \cite[Thm.~6, p.87]{gantmacher2002}
  it follows that the eigenvalues of the collocation matrix
  \eqref{eq:ISdk-collocation} are
  distinct positive real numbers, i.e.,
  \(\lambda_1 > \lambda_2 > \cdots > \lambda_n > 0\).  
\end{proof}

Now, we can use this property to show that the eigenvalues of the Schoenberg
operator are non-negative, real numbers. Additionally, the only eigenvalue with
multiplicity two is $1$, whereas all the others have multiplicity
one.\index{Schoenberg operator!eigenvalues}
\begin{theorem}
  \label{thm:schoenberg-eigenvalues-distinct}
  The eigenvalues of the Schoenberg operator are characterized by
  \begin{equation*}
    1 = \lambda_{0} = \lambda_{1} > \lambda_{2} > \cdots > \lambda_{n+k-1} > \lambda_{n+k} = 0.
  \end{equation*}
  Thus, besides \(0\) and\, \(1\) the Schoenberg operator has \(n+k-1\) distinct positive real eigenvalues. 
\end{theorem}
\begin{proof}
  We use that $\ISdkm{1}$ has $n+k-1$ distinct positive eigenvalues 
  combined with the eigenvalue $0$ coming from the finite-dimensional range of $\ISdkm{1}$
  and \autoref{lem:relation_sdk_isdk} saying that
  \begin{equation*}
    D\Sdk f = \ISdkm{1} D f
  \end{equation*}
  holds for all $f \in \spacecf$.

  We show first that $0 \in \sigma_p(\Sdk)$. To this end, let $f \in \spacecf$ be a
  function, such that
  \begin{equation*}
    f(\xi_j) = 0 \qquad \text{for all } j \in \setof{-k, \ldots, n-1} 
  \end{equation*}
  and such that there exists $x \in \ivcc{0}{1} \setminus \cset{\xi_j}{j \in
    \setof{-k,\ldots,n-1}}$ with $f(x) \neq 0$.  For example,
  consider the polynomial $f(x) =
  \prod_{i=-k}^{n-1}(x-\xi_i)$. Clearly, $f \in \spacecf$ and we
  obtain $\Sdk f = 0 \cdot f = 0$, because for all $x \in \ivcc{0}{1}$
  \begin{equation*}
    \Sdk f(x) = \sum_{j=-k}^{n-1}\left[\prod_{i=-k}^{n-1}(\xi_j-\xi_i)\right]\bspk{j}(x) = 0.
  \end{equation*}

  We now construct the set of eigenvalues and eigenfunctions of $\Sdk$ by their relation
  to the integral Schoenberg operator $\ISdkm{1}$.
  To this end, let us consider now an eigenfunction $s \in \spacepp$ of $\Sdk$ 
  corresponding to some eigenvalue $\lambda \in \sigma_p(\Sdk) \setminus \setof{0,1}$. 
  Then we calculate
  \begin{equation*}
   \ISdkm{1} D s = D\Sdk s = \lambda D s.
  \end{equation*}
  This states in particular 
  that the eigenvalue $\lambda \neq 0$ of the Schoenberg operator 
  with corresponding eigenfunction $s$ is again an eigenvalue of $\ISdkm{1}$ with
  associated eigenfunction $Ds$. 
  The only exception yields the eigenfunction $1$. Here, we obtain
  \begin{equation*}
    \ISdk D 1 = D 1 = 0.
  \end{equation*}
  Therefore, $0 = D 1$ does not yield a new linear independent eigenfunction of $\ISdkm{1}$.
  Whereas, the eigenfunction $x$ corresponding to the eigenvalue $1$ is mapped to the constant
  eigenfunction $1$:
  \begin{equation*}
    \ISdk D x = D \Sdk x = D x = 1.
  \end{equation*}
  Now we use that all eigenfunctions $s_1,\ldots,s_{n+k-1}$ of $\Sdk$
  corresponding to the eigenvalues $\lambda_1,\ldots,\lambda_{n+k-1}$ are
  linearly independent, to conclude that the same holds true for the functions
  $Ds_1,\ldots,Ds_{n+k-1}$.  
  Consequently, the positive numbers $\lambda_1,\ldots,\lambda_{n+k-1}$ are exactly the $n+k-1$
  distinct eigenvalues of $\Sdk$.
\end{proof}

\bibliographystyle{plainnat}
\bibliography{references.bib}

\end{document}